\documentclass[pdftex]{amsart}

\usepackage{amssymb,amsmath,amsthm,bm,amscd}
\usepackage[pdftex]{graphicx}

\usepackage{hyperref}
\hypersetup{colorlinks,allcolors=black}

\newtheorem{theorem}{Theorem}[section]
\newtheorem{lem}[theorem]{Lemma}
\newtheorem{cor}[theorem]{Corollary}
\newtheorem{proposition}[theorem]{Proposition}

\theoremstyle{definition}
\newtheorem{definition}[theorem]{Definition}
\newtheorem{example}[theorem]{Example}

\newtheorem{hypo}[theorem]{Hypothesis}

\theoremstyle{remark}
\newtheorem{remark}[theorem]{Remark}

\numberwithin{equation}{section}

\begin{document}

\title[Information geometry and barycenter maps]{Information geometry of the space\\
of probability measures\\
and barycenter maps}
\footnote[0]{This article is published in Sugaku Expositions \textbf{34} (2021), 231--253,
originally appeared in Japanese in S\={u}gaku \textbf{69} 4 (2017), 387--406.}

\author{Mitsuhiro Itoh}
\address{University of Tsukuba,
Institute of Mathematics,
1-1-1 Ten-no-dai,
Tsukuba City, 305-8571, Japan}
\curraddr{}
\email{itohm@math.tsukuba.ac.jp}
\thanks{}

\author{Hiroyasu Satoh}
\address{Nippon Institute of Technology,
Liberal Arts and Sciences, 4-1 Gakuendai, Miyashiro-machi, Saitama Prefecture, 345-8501 Japan}
\curraddr{}
\email{hiroyasu@nit.ac.jp}
\thanks{}

\subjclass[2020]{Primary 46E27; Secondary 58B20, 53B21.}

\date{}
\dedicatory{}

\begin{abstract}
In this article, we present recent developments of information {geometry}, namely, geometry of the Fisher metric, dualistic structures and divergences on the space of probability measures,
particularly the theory of geodesics of the Fisher metric.
Moreover, we consider several facts concerning the barycenter of probability measures on the ideal boundary of a Hadamard manifold from a viewpoint of the information geometry.
\end{abstract}

\maketitle

\section{Introduction}\label{intro}

We present recent progress of information geometry, geometry with respect to the Fisher metric, the dual connection structure and divergences for a space of probability distributions. Especially we deal with geodesics with respect to the Fisher metric. Moreover, we discuss a barycenter map defined for probability measures which are defined on an ideal boundary of a Hadamard manifold.

Let $A_1, A_2, \cdots, A_n$ be points of the Euclidean space $\mathbb{E}$ and $\bm{w}=(w_1,\ldots, w_n)$ be an ordered $n$-tuple of non-negative numbers with $\sum_{i=1}^n w_i = 1$.
Then, we call a point $P\in\mathbb{E}$ satisfying $\overrightarrow{OP} = \sum_{i=1}^n w_i \overrightarrow{OA}_i$ the barycenter of $A_i$, $i=1, 2,\cdots, n$ with weight $\bm{w}$.
Here $O$ is the origin of $\mathbb{E}$.
We remark that the barycenter $P$ is independent of the choice of a reference point $O$.
Consider for example $n=3$ and $w_i = 1/3$, $i=1,2,3$.
We may set $O = A_1$. Then the barycenter $P$ of points $A_1, A_2, A_3$ satisfies $\overrightarrow{A_1P}=\frac{1}{3} (\overrightarrow{A_1 A_2} + \overrightarrow{A_1 A_3})$.
This means that $P$ coincides with the center of gravity of the triangle $\triangle A_1A_2A_3$. In fact, $\overrightarrow{A_1P}$ equals the vector, multiplied by $2/3$, of the geometric vector from $A_1$ to the midpoint of its opposite side $\frac{1}{2}\left(\overrightarrow{A_1A_2} + \overrightarrow{A_1 A_3}\right)$.
We can also define the barycenter by a critical point of the function  $f : \mathbb{E} \rightarrow \mathbb{R} ; Q \mapsto \sum_{i=1}^n w_i \vert A_i - Q\vert^2$.
In case of points being continuously distributed, for a non-negative function $w=w(x)$ satisfying $\int_\mathbb{E} w(x)\,dx = 1$,
we define the function $f : \mathbb{E} \rightarrow \mathbb{R}\, ; y \mapsto \int_{x\in {\Bbb E}} d(y, x)^2 w(x)\,dx$ and call  a critical point of $f$ a barycenter with weight $w= w(x)$, where $d(\cdot,\cdot)$ is the distance of $\mathbb{E}$.
Since the weight function $w= w(x)$ is regarded as  a density of substance distributed on $\mathbb{E}$ of unit total mass, or density function of  a probability distribution,
a barycenter can be defined for a probability measure $w(x)\,dx$. The notion of barycenter contributes to the conjugacy theorem of maximal compact subgroups of a semi-simple Lie group, which is one of theorems necessary for the theory of symmetric spaces, and is a consequence of the Cartan fixed point theorem in which the barycenter is utilized(\cite{Hel-1, Eb'96}).
The convexity of distance function on a Riemannian manifold of non-positive curvature plays an important role in investigation of barycenter.
In this article, we consider barycenters with respect to a convex function, namely the Busemann function, in place of the distance function.

On considering barycenters, we need to study a space of probability measures.
Let $(M, g)$ be a Riemannian manifold and $\mathcal{P}^+(M)$ be the space consisting of all probability measures having positive density function.
We can define a Riemannian metric, called the Fisher metric on $\mathcal{P}^+(M)$.
The Fisher metric provides a positive definite inner product to each tangent space $T_{\mu}\mathcal{P}^+(M)$.
Riemannian geometry of the Fisher metric on a space of probability measures has been established by T. Friedrich\cite{Fr}.
In \cite{Fr}, explicit formulae of the Levi-Civita connection and geodesics of the Fisher metric are obtained in terms of density functions.
From the representation form of geodesics, we find that all geodesics of the Fisher metric are periodic, whereas $\mathcal{P}^+(M)$ is not geodesically complete.
Moreover, it is shown that the Fisher metric is  a metric of constant sectional curvature $1/4$.
We develop the information geometry of the Fisher metric based on T. Friedrich's work  and  apply it to the geometry of barycenter maps.
We can argue the theory of geodesics, one of the basic subjects and tools of Riemannian geometry, on the space $(\mathcal{P}^+(M), G)$.
We define in this article the geometric mean of two probability measures.
This notion gives a good understanding about shortest geodesics, the exponential map, the distance function and thus it enables us to develop the theory of geodesics elaborately.

The Fisher metric is a natural generalization of the Fisher information matrix in mathematical statistics and information theory.
As is commonly known, a family of connections $\{\nabla^{(\alpha)}\}_{\alpha\in\mathbb{R}}$ in which a pair $(\nabla^{(\alpha)}, \nabla^{(-\alpha)})$ is dual with respect to the Fisher metric plays a significantly important role in the geometry of statistical manifolds. Here $\nabla^{(\alpha)}$ is called the $\alpha$-connection (see \cite{AN}).

We can formulate the information geometry of $\alpha$-connections on a statistical manifolds more widely on the space $\mathcal{P}^+(M)$, in particular we can develop on the space $\mathcal{P}^+(M)$ the information geometry of the specific $\alpha$-connections, namely, $(+1)$-connection $\nabla^{(+1)}$ and $(-1)$-connection $\nabla^{(-1)}$, which are flat and dual each other (see Corollary \ref{flatdual}, \S \ref{alpha} and \cite{I}).
With respect to this fact, a family of straight lines $t \mapsto \mu + t \tau$ permits an interpretation of a family of geodesics of $m$-connection, namely, the $(-1)$-connection, which corresponds to an affine coordinate system in a statistical manifold.
By this affine parametrization, we observe that the Kullback-Leibler divergence, one of the important entropies in information geometry, provides a potential by which the Fisher metric is considered as a Hesse metric.

On the other hand, Besson et al. (\cite{BCG}) showed  Mostow's rigidity theorems for compact manifolds of negative curvature by using the geometric quantity called the volume entropy 
 and by applying the notion of barycenter which is initiated by Douady-Earle (\cite{DE}) on the hyperbolic plane. In their study Besson et al. considered probability measures, including the Poisson kernel measure, which are defined on the ideal boundary $\partial X \cong S^{n-1}$ of a simply connected negatively curved manifold $X^n$. They formulate for those measures the barycenter whose test function is the Busemann function. 
They deal with probability measures without atom, not restricting themselves to probability measures of positive density function.
Here a probability measure $\mu$ has no atom, when for any Borel set $A$ with $\mu(A) > 0$ there exists a Borel set $B \subset A$ satisfying $0<\mu(B) < \mu(A)$. Their consideration is restrictive within rank one symmetric spaces of non-compact type.
In this article, we develop, however, geometry of barycenter on a Hadamard manifold $X$ more generic than symmetric spaces of non-compact type.
Let $\partial X$ be the ideal boundary of $X$ which is defined as a quotient of all geodesic rays by asymptotic equivalence relation and $B_{\theta}$ be the normalized Busemann function.
For a probability measure $\mu$ on $\partial X$, we consider the averaged  Busemann function with weight $\mu$,
$$
\mathbb{B}_{\mu}: X \rightarrow \mathbb{R} ;\ y \mapsto \int_{\theta\in\partial X} B_{\theta}(y)\,d\mu(\theta).
$$
We call a critical point $x\in X$ of $\mathbb{B}_{\mu}$ the barycenter of $\mu$.
Under the guarantee of existence and uniqueness of barycenter, we can define a map ${\rm bar} : \mathcal{P}^+(\partial X) \rightarrow X$, by assigning to $\mu$ a point $x\in X$ which is a barycenter of $\mu$
and we call ${\rm bar}$ the barycenter map (see also \cite[p.554]{E}).
We discuss the existence and uniqueness of barycenter for arbitrary $\mu$ in \S \ref{barycenter}. One of important properties of the barycenter map is the following:
an isometry $\varphi$ of $X$ induces naturally a homeomorphism ${\hat \varphi} : \partial X \rightarrow \partial X$ which satisfies $\varphi\circ {\rm bar}={\rm bar}\circ {\hat \varphi}$.

Now we will explain the motivation and background of our investigations of barycenters.
An $n$-dimensional Hadamard manifold $(X, g)$ is diffeomorphic to ${\Bbb R}^n$ and hence to an open ball $D^n$ with boundary $S^{n-1}$. A Hadamard manifold 
$X$ admits also the ideal boundary $\partial X$.
Then, we are able to consider the Dirichlet problem at infinity: given an $f \in C^0(\partial X$),
find a solution $u = u(x)$ on $X$ to $\Delta u\vert_X=0, u\vert_{\partial X}=f$,
where $\Delta$ is the Laplace-Beltrami operator.
This is a geometric extension of the classical Dirichlet problem on a given bounded region with boundary in ${\Bbb R}^n$ to a Hadamard manifold with ideal boundary.
Now, we suppose that a solution $u$ of the Dirichlet problem at infinity with boundary condition is described in terms of the integral kernel $P(x,\theta)$, called Poisson kernel, as $u(x)=\int_{\theta\in\partial X}P(x,\theta) f(\theta)\,d\theta$.
Then, the fundamental solution $P(x,\theta)$ together with the measure $d\theta$ gives a probability measure on $\partial X$ with positive density function (see \S \ref{barycentermap} and refer to \cite{SY,AS,E} for precise definition of Poisson kernel).
When the existence of the Poisson kernel is guaranteed, we can define a map $\Theta:X\to\mathcal{P}^+(\partial X);\,\,x\mapsto P(x,\theta)\,d\theta$, which we call the Poisson kernel map.
As mentioned above, $\mathcal{P}^+(\partial X)$ carries the Fisher metric and hence, the map $\Theta$ is regarded as an embedding from a Hadamard manifold $(X, g)$ into an infinite dimensional Riemannian manifold $(\mathcal{P}^+(\partial X), G)$ of constant curvature.
Then, we obtain 

\begin{theorem}[\cite{ISh, IS1, IS2}]\label{poissonmap}\ 
\begin{enumerate}
\item Let $(X^n,g)$ be an $n$-dimensional Damek-Ricci space.
Then, $(X^n,g)$ carries a Poisson kernel.
Moreover, its Poisson kernel map $\Theta$ is a  homothety, i.e., satisfies $\Theta^{\ast}G= C\, g$, $C = Q/n$ and is harmonic (i.e., minimal).
Here $Q>0$ denotes the volume entropy of $X$ which means the exponential volume growth of $X$.
\item Conversely, if a Hadamard manifold $(X^n,g)$ admits a Poisson kernel and its Poisson kernel map $\Theta$ is homothetic and harmonic, then the Poisson kernel of $X$ is expressed in the form
\begin{equation}
P(x,\theta)=\exp(-Q\,B_{\theta}(x)),
\end{equation}
where $B_{\theta}(x)$ is the normalized Busemann function.
\end{enumerate}
\end{theorem}

In this article, we call the Poisson kernel represented in Theorem \ref{poissonmap} (ii), the Busemann--Poisson kernel. See \S \ref{barycentermap} for details.
A Hadamard manifold which carries the Busemann--Poisson kernel satisfies visibility axiom and is asymptotically harmonic. We say that $X$ satisfies the visibility axiom, if there exists a geodesic in $X$ joining arbitrary distinct two ideal points (see \S \ref{barycenter}). Moreover, we say that $X$ is asymptotically harmonic, if all horospheres, level hypersurfaces of the Busemann function $B_{\theta}$ for any $\theta\in\partial(X)$ have constant mean curvature $c$ (\cite{Le}).
We remark that if $X$ is asymptotically harmonic and all horospheres have constant mean curvature $c$, then $c$ coincides $-Q$ (\cite{ISS}).
When $X$ carries the Busemann--Poisson kernel, the barycenter of the probability measure $P(x,\theta)\,d \theta =\exp(-Q\,B_{\theta}(x))\,d \theta$ is just the point $x$.
When we regard the barycenter map ${\rm bar} : \mathcal{P}^+(\partial X) \rightarrow X$ as the projection of a fiber space structure, the Poisson kernel map $\Theta : X \rightarrow \mathcal{P}^+(\partial X)$ provides a section of the projection ${\rm bar}$.

In connection with the Busemann-Poisson kernel we make a brief introduction of Damek-Ricci spaces.
A Damek-Ricci space is a solvable Lie group carrying a left-invariant metric. The family of all Damek-Ricci spaces is a class of harmonic Hadamard manifolds including rank one symmetric spaces, i.e., the complex, quaternionic hyperbolic $n$-spaces 
$\mathbb{C}{\bf H}^n$, $\mathbb{H}{\bf H}^n$, $n \geq 1$, and the octonionic hyperbolic plane $\mathbb{O}{\bf H}^2$. 
The real hyperbolic spaces are regarded as special ones among Damek-Ricci spaces.
Refer to \cite{BTV,ADY,Kashi} for Damek-Ricci spaces.
Here we call a Riemannian manifold harmonic, if the volume density function or mean curvature of any geodesic spheres depends only on radius and independent of the direction from the center (\cite{Sz}).
Since a horosphere is certain limit of geodesic spheres, harmonic manifolds are also asymptotically harmonic.
We notice that Damek-Ricci spaces carry the Busemann--Poisson kernel (\cite{ADY, IS2}).
Since a Hadamard $(X,g)$ carrying the Busemann--Poisson kernel is asymptotically harmonic, one has the problem whether such a manifold is harmonic or not and, moreover, one can raise the problem whether it is isometric or homothetic to a Damek-Ricci space.
We take one further step and suppose that such a manifold $(X,g)$ is quasi-isometric to a Damek-Ricci space $S$.
Then, under this assumption, we can raise a problem whether there exists an isometry or a homothety between $(X,g)$ and the space $S$.
We see that from this assumption an isometry of $S$ induces a quasi-isometry of $(X,g)$ and hence a homeomorphism of the ideal boundary $\partial X$ of $X$, since any Damek-Ricci space is Gromov hyperbolic.
With these backgrounds, we proceed to our argument of barycenter of probability measures on the ideal boundary.

The article is organized as follows.
In \S \ref{2}, we treat differential geometry of the Fisher metric $G$ , i.e., the Levi-Civita connection of the metric $G$ on the space of probability measures.
In \S \ref{alpha}, we introduce an outline of the theory of $\alpha$-connections which admits duality with respect to the Fisher metric.
In \S \ref{4}, we recall the basic properties and the important results of the ideal boundary of a Hadamard manifold and the normalized Busemann function that are needed for  discussing barycenters in the later sections.
We define the barycenter in \S \ref{barycenter} and consider its existence and uniqueness, and then its geometric properties.
In the last section, we develop our argument of the fiber space structure of the barycenter map and isometricity of a barycentrically associated map that is a transformation of $X$ induced by a homeomorphism of $\partial X$ via the barycenter map.

\section{The space of probability measures and the Fisher metric}\label{2}

\subsection{The space of probability measures}\label{2.1}

Let $M$ be a compact, connected  $C^{\infty}$-manifold and $\mathcal{B}(M)$ be the collection of all Borel sets on $M$.
$\mathcal{B}(M)$ is the smallest $\sigma$-algebra which contains all open subsets of $M$.

A probability measure on the measurable space $(M,\mathcal{B}(M))$, or simply on $M$, is a real valued function $P:\mathcal{B}(M)\to\mathbb{R}$ satisfying the following:
\begin{enumerate}
\item $P(A)\geq 0$ holds for any $A\in\mathcal{B}(M)$. 
In particular, $P(M)=1$ and $P(\emptyset)=0$.
\item For any countable sequence of sets 
$\{E_j\,\vert\,j=1,2,\cdots\}$ of $\mathcal{B}(M)$   satisfying $E_j\cap E_k=\emptyset,\ j\not=k$ 
\begin{equation}\label{21}
P\left(\mathop{\bigcup}_{j=1}^{\infty}E_j\right)=\sum_{j=1}^{\infty}P(E_j).
\end{equation}
\end{enumerate}

Let  $\lambda$ be a volume form on $M$.
We normalize $\lambda$ such that $\int_{x\in M}d\lambda(x)=1$ and fix it as the standard probability measure on $M$.

Let $\mathcal{P}(M)$ be the space of all probability measures on $M$. Obviously
$\lambda\in\mathcal{P}(M)$.
Any probability measure $\mu$ which we consider in this article is supposed to be absolutely continuous with respect to $\lambda$, denoted as $\mu \ll \lambda$.
Here, a probability measure $\mu\in \mathcal{P}(M)$ is said to be absolutely continuous with respect to $\mu_1\in \mathcal{P}(M)$, provided $\mu(A)=0$ holds for any $A\in\mathcal{B}(M)$ satisfying $\mu_1(A) = 0$.
A measure $\mu$ is absolutely continuous with respect to a finite measure $\mu_1$ if and only if there exists a function $f\in L^1(M,\mu_1)$ such that $\mu=f \mu_1$, i.e.,
\begin{equation*}
\mu(A)=\int_{x\in A}f(x) d\mu_1(x)\quad(\forall\,A\in\mathcal{B}(M)).
\end{equation*}
We call such a function $f\in L^1(M,\mu_1)$ the Radon-Nikodym derivative of $\mu$ with respect to $\mu_1$, denoted by $d\mu/d\mu_1$. Probability measures which we consider in this article are measures whose Radon-Nikodym derivative with respect to $\lambda$ is everywhere positive and continuous.
We denote the space of all such probability measures by $\mathcal{P}^+(M)$ and call it the space of probability measures:
\begin{equation}\label{24}
\mathcal{P}^+(M)=\left\{\mu\in \mathcal{P}(M)\,\left\vert\,\mu \ll \lambda,\,\,\,\frac{d\mu}{d\lambda}\in C^0(M),\,\,\,
\frac{d\mu}{d\lambda}(x)>0\,\,(\forall\, x\in M)\right.\right\}.
\end{equation}
Then, a natural embedding is defined from $\mathcal{P}^+(M)$ into the space of $L^2$-functions $L^2(M,\lambda)$:
\begin{equation}\label{27}
\rho = \rho^{(1/2)}\ :\, 
\mathcal{P}^+(M)  \hookrightarrow  L^2(M, \lambda);
\mu = f\,\lambda  \mapsto  \displaystyle{2\, \sqrt{\frac{d \mu}{d \lambda}} = 2\ \sqrt{f}}.
\end{equation}
We define a topology of $\mathcal{P}^+(M)$ by this embedding.
We remark that there exists 
a sequence $\{\mu_i\}$ of probability measures of $\mathcal{P}^+(M)$ which has not necessarily a limit in $\mathcal{P}^+(M)$, even if the sequence $\{\sqrt{d\mu_i/d\lambda}\}$ is convergent in $L^2(M,\lambda)$, 
We notice that geometry of infinite dimensional space of probability measures is discussed in \cite{PS}.

Each probability measure in $\mathcal{P}^+(M)$ is regarded as a point in a space.
A path joining $\mu = f\,\lambda, \mu_1= f_1\,\lambda\in\mathcal{P}^+(M)$ given by 
$\mu(t)=(1-t)\mu+t\mu_1 = \left((1-t) f + t f_1\right) \lambda$ is inside $\mathcal{P}^+(M)$ for $0 \leq t \leq 1$.
Differentiating this curve $\mu(t)$ with respect to $t$,  we have
\begin{equation*}\label{30}
\dot{\mu}(t)=\frac{d}{dt}((1-t)\mu+t\mu_1)=\mu_1-\mu = \left(f_1 - f \right) \lambda
\end{equation*}
and then
$$
\int_M d\dot{\mu}(t)=0.
$$
From this consideration, we define a tangent space $T_{\mu}\mathcal{P}^+(M)$ at $\mu\in\mathcal{P}^+(M)$ as follows:
\begin{equation}\label{32}
T_{\mu}\mathcal{P}^+(M) := \left\{ \tau = h\,\lambda\, |\, \, h \in C^0(M)\cap L^2(M,\lambda),\, \int_{x\in M} h(x)\,d\lambda(x) = 0 \right\}.
\end{equation}

\begin{remark}
A path joining two points is a 1-simplex.
In general, an $n$-simplex $\Delta^n$ whose vertices of number $n+1$ are probability measures of $\mathcal{P}^+(M)$ is a proper subset of $\mathcal{P}^+(M)$.
\end{remark}

\begin{remark}
The right hand side of \eqref{32} is an infinite dimensional vector space whose definition is independent of the choice of $\mu$.
We denote therefore this space by $\mathcal{V}$.
Let $\mu\in\mathcal{P}^+(M)$. Then the sum $\mu+\tau$  of $\mu$ with $\tau\in\mathcal{V}$ belongs to $\mathcal{P}^+(M)$, if  the $C^0$-norm $\sup_{x\in M} \left|(d\tau/d\mu)(x)\right|$ of $\tau$ with respect to $\mu$ is small enough.
This fact suggests that $\mathcal{P}^+(M)$ is like an open subset in an affine space whose associated vector space is $\mathcal{V}$.
\end{remark}

We define a curve $c:(a,b)\to\mathcal{P}^+(M)$ in $\mathcal{P}^+(M)$ by
\begin{equation*}\label{33}
c(t)=f(x,t)\,\lambda\qquad(a < t < b).
\end{equation*}
We assume that the density function $f(x,t)$ parameterized in $t$ is of $C^2$-class with respect to $t$ for each $x\in M$.
The velocity vector of the curve $c(t)$ is given by
\begin{equation*}\label{34}
\frac{dc}{dt}(t)=\frac{\partial f}{\partial t}(x,t)\,\lambda\, \in\, T_{c(t)}\mathcal{P}^+(M),\ t\in(a,b).
\end{equation*}

\subsection{Fisher Metric}\label{2.2}

We give definition of the Fisher metric and state its properties.
\begin{definition}[\cite{Fr, I, IS2, IS5, IS6, ISh, Sh}]
Let $\mu=f(x)\,\lambda\in\mathcal{P}^+(M)$.
We define a positive definite inner product $G_\mu$ on $T_{\mu}\mathcal{P}^+(M)$ by
\begin{equation}\label{35}
 G_{\mu}(\tau,\tau_1) = \int_{x\in M}
\frac{d\tau}{d\mu}(x)\,\frac{d\tau_1}{d\mu}(x)\,d\mu(x)
=\int_{x\in M}
\frac{h(x)}{f(x)}\frac{h_1(x)}{f(x)}f(x)\,d\lambda(x)
\end{equation}
for $\tau= h(x)\,\lambda,\ \tau_1 = h_1(x)\,\lambda\in T_{\mu}\mathcal{P}^+(M)$.
We call a map $\mu\mapsto G_\mu$ the Fisher metric on $\mathcal{P}^+(M)$.
We denote $G_{\mu}$-norm of $\tau$ by $\vert \tau \vert_{\mu}:= \sqrt{G_{\mu}(\tau,\tau)}$.
\end{definition}

See \cite{PS} also for the Fisher metric defined on an infinite dimensional space of probability measures. 

Now, we give a push-forward measure which is basic in measure theory and also in probability theory.
Let $\Phi$ be a homeomorphism of $M$.
A push-forward of measures by $\Phi$ is a map $\Phi_{\sharp} : \mathcal{P}^+(M) \to \mathcal{P}^+(M)$ which assigns to any $\mu\in \mathcal{P}^+(M)$ a measure $\Phi_{\sharp}\mu\in \mathcal{P}^+(M)$. Here $\Phi_{\sharp}\mu$ is defined as follows: for any $A\in\mathcal{B}(M)$, $\displaystyle{\Phi_{\sharp}(\mu)(A) := \mu(\Phi^{-1}(A))}$, i.e., $\Phi_{\sharp}\mu$ is a probability measure which satisfies
\begin{equation}\label{41}
\int_{x\in M}h(x)\,d(\Phi_{\sharp}\mu)(x)
=\int_{x'\in\Phi^{-1}(M)}h(\Phi(x'))\,d\mu(x')
\end{equation}for any measurable function $h:M\to\mathbb{R}$. When $\Phi : M\to M$ is a diffeomorphism of $M$,
we find $\Phi_{\sharp}\ \mu=(\Phi^{-1})^{\ast} \mu$, i.e., the push-forward $\Phi_{\sharp}$ is the pullback of measures by the inverse diffeomorphism $\Phi^{-1}$.
Notice that the differential map of the push-forward $\Phi_{\sharp}$ is given by $(d\Phi_{\sharp})_{\mu}(\tau)=\Phi_{\sharp}(\tau)$.

\begin{theorem}[\cite{Fr}]\label{pushisometry}
The push-forward by a homeomorphism $\Phi\,:\,  M \rightarrow M$ acts isometrically on $(\mathcal{P}^+(M), G)$,
i.e., for any $\tau,\tau_1\in T_{\mu}\mathcal{P}^+(M)$ and $\mu\in\mathcal{P}^+(M)$

\begin{equation}\label{pushforwardisometry}
G_{\Phi_{\sharp}\mu}\left((d\Phi_{\sharp})_{\mu}(\tau),\,(d\Phi_{\sharp})_{\mu}(\tau_1)\right)
=G_{\mu}(\tau,\,\tau_1). 
\end{equation}
\end{theorem}

\begin{remark}
It is known that for any $\mu\in\mathcal{P}^+(M)$ there exists a homeomorphism $\Phi:M\to M$ satisfying $\Phi_{\sharp}\lambda=\mu$ (see \cite{Fa, OU}), i.e., the group of homeomorphisms of $M$ acts on $\mathcal{P}^+(M)$ transitively.
Hence, we find that $({\mathcal P}^+(M),G)$ is a Riemannian homogeneous space.\end{remark}

The embedding $\rho:\mathcal{P}^+(M)\hookrightarrow L^2(M,\lambda)$ defined at \eqref{27} satisfies the following.
\begin{proposition}
The pull-back of the $L^2$-inner product $\langle \cdot, \cdot\rangle_{L^2}$ on $L^2(M,\lambda)$ by $\rho$ coincides with the Fisher metric $G$, i.e.,
for $\tau, \tau_1\in T_{\mu}\mathcal{P}^+(M)$, $\mu\in\mathcal{P}^+(M)$ we have
\begin{equation*}
\left\langle d\rho_{\mu}\tau,d\rho_{\mu}\tau_1\right\rangle_{L^2}
= G_{\mu}(\tau,\tau_1).
\end{equation*}
\end{proposition}

\begin{remark}
Since the image of the embedding $\rho$ satisfies ${\rm Im}\ \rho \subset S^{\infty}(2) = \{ h\in L^2(M,\lambda)\, \vert\, \vert h\vert_{L^2} = 2\}$,
we can realize the Riemannian geometry of $({\mathcal P}^+(M),G)$ as an extrinsic submanifold geometry inside the infinite dimensional sphere $S^{\infty}(2)$ of radius $2$.
\end{remark}

\begin{remark}
The Fisher metric $G$ is obtained also as the second derivative of the Kullback-Leibler divergence
\begin{equation*}
D_{\mathrm{KL}} : \mathcal{P}^+(M) \times \mathcal{P}^+(M) \rightarrow {\mathbb R};\,
D_{\mathrm{KL}}(\mu\vert\vert\mu_1) = -\int_{x\in M} \log \left(\frac{d\mu_1}{d\mu}(x)\right) d\mu(x),
\end{equation*}
i.e., one observes $\left.\frac{d^2}{d t^2}\right|_{t=0} D_{\mathrm{KL}}(\mu\vert\vert\mu + t \tau) = G_{\mu}(\tau,\tau)$.
See \cite{AN} for the arguments for statistical manifolds.
$D_{\mathrm{KL}}$ is called  a potential in Hesse geometry (\cite{Shima}).
\end{remark}

Since the Fisher metric $G$ is regarded as a Riemannian metric on $\mathcal{P}^+(M)$, the Levi-Civita connection  is uniquely determined for the metric $G$.
By the method of constant vector fields employed by T. Friedrich \cite{Fr}, we can express the Levi-Civita connection explicitly.

For any vector $\tau\in\mathcal{V}$, we define a constant vector field on $\mathcal{P}^+(M)$ as follows:
$\tau_{\mu} := \tau \in  T_{\mu}\mathcal{P}^+(M)$ at  $\mu\in\mathcal{P}^+(M)$.
The integral curve of the constant vector field $\tau$ starting from $\mu\in\mathcal{P}^+(M)$ is given by $t\mapsto\mu+t\tau$.

\begin{theorem}[\cite{Fr}]\label{Levicivita}

Let $\tau, \tau_1\in \mathcal{V}$ be vectors regarded as constant vector fields on $\mathcal{P}^+(M)$. Then, the Levi-Civita connection $\nabla$ of the Fisher metric $G$ is represented by
 \begin{equation}\label{54}
 (\nabla_{\tau}\tau_1)_{\mu}
=-\frac{1}{2}\left(\frac{d\tau}{d\mu}\frac{d\tau_1}{d\mu}-G_{\mu}(\tau,\tau_1)\right)\mu.
\end{equation}
\end{theorem}

We remark that the right hand side of (\ref{54}) gives an element of $\mathcal{V}$.

On the Riemannian curvature tensor of the metric $G$ we have the following.
\begin{theorem}[\cite{Fr, I}]\label{Riemanncurvature}
The Riemannian curvature tensor of the Fisher metric $G$ is given as follows:
\begin{equation}\label{55}
R_{\mu}(\tau_1,\tau_2)\tau=\frac{1}{4}\left(G_{\mu}(\tau,\tau_2)\tau_1-G_{\mu}(\tau,\tau_1)\tau_2\right),
\end{equation}
which indicates that $(\mathcal{P}^+(M),G)$ is a Riemannian manifold of constant sectional curvature $1/4$.
\end{theorem}

The following theorem gives a description of geodesics associated to the Levi-Civita connection.
\begin{theorem}[\cite{I, IS6, IS5}]\label{geodesic}
Let $t\mapsto\mu(t)= f(x,t)\,\lambda$ be a geodesic on $\mathcal{P}^+(M)$ with respect to the Levi-Civita connection $\nabla$ with initial conditions $\mu(0)=\mu$ and $\dot{\mu}(0) = \tau \in T_{\mu}\mathcal{P}^+(M)$, $\vert \tau\vert_{\mu} =1$.
Then, $\mu(t)$ is described by
\begin{equation}\label{56}
\mu(t)=\left(\cos\frac{t}{2}+\sin\frac{t}{2}\frac{d\tau}{d\mu}\right)^2\mu.
\end{equation}
\end{theorem}

As a consequence, it is shown that the geodesic $\mu(t)$ is periodic of period $2\pi$, while $(\mathcal{P}^+(M), G)$ is not geodesically complete.
In fact, the value of $\mu(t)$ at $t= \pi$ is given  from (\ref{56}) by $\mu(\pi)=\left(d\tau/d\mu\right)^2\mu$, whereas there exists $x_0\in M$ such that $d\tau/d\mu(x_0)=0$, since $\tau$ satisfies $\int_{x\in M} \left(d\tau/d\mu\right)d\mu(x) = \int_{x\in M} d\tau(x) = 0$.
Hence, the density function of a probability measure $\mu(\pi)$ is never positive at least at $x_0$ and then $\mu(\pi)\not\in \mathcal{P}^+(M)$. We find thus that any geodesic $\mu(t)$ can not reach $\mu(\pi)$ inside $\mathcal{P}^+(M)$.

Theorem \ref{geodesic} is shown by the lemmas given by Friedrich (see \cite{Fr, I}), while the detail of the proof is omitted.
We give now definition of a specific function together with a certain map for intimate investigation of geodesics associated to the Fisher metric.

\begin{definition}\label{geometricmean}
Let $\mu$, $\mu_1\in\mathcal{P}^+(M)$.
Define a function $\ell : \mathcal{P}^+(M)\times \mathcal{P}^+(M)\rightarrow [0,\pi)$ and a map $\sigma : \mathcal{P}^+(M)\times \mathcal{P}^+(M)\rightarrow \mathcal{P}^+(M)$, respectively, as follows:
\begin{equation}
\cos \frac{\ell(\mu,\mu_1)}{2} := \int_{x\in M} \sqrt{\frac{d \mu_1}{d\mu}}(x) d \mu(x),  \end{equation}
\begin{equation}
{ \sigma}(\mu,\mu_1):= \, \left(\cos \frac{\ell(\mu,\mu_1)}{2}\right)^{-1}\, \sqrt{\frac{d \mu_1}{d \mu}} \,\mu.
\end{equation}
We call $\sigma(\mu,\mu_1)$ a normalized geometric mean of $\mu$ and $\mu_1$.
\end{definition}

Then, we have the following.
\begin{theorem}[\cite{I,IS6}]\label{geodesic-x}
For $\mu$,\, $\mu_1$, there exists a unique geodesic $\mu(t)$ joining $\mu$ and $\mu_1$ which satisfies $\mu(0) = \mu$, $\mu(\ell) = \mu_1$ and $\vert \dot \mu(0)\vert_{\mu} = 1$ and $\mu([0,\ell]) \subset {\mathcal P}^+(M)$.
Moreover,  $\ell = \ell(\mu,\mu_1)$.
\end{theorem}

In fact, let $\ell = \ell(\mu,\mu_1)$. Then
\begin{equation}
\mu(t) = \left(\cos \frac{t}{2} + \sin \frac{t}{2}\ \frac{d\tau}{d\mu} \right)^2 \mu,
\qquad\tau = \frac{1}{\tan (\ell/2)}\, \left({\sigma}(\mu, \mu_1)- \mu\right)
\end{equation}
is the unique geodesic satisfying $\mu(0) = \mu$, $\mu(\ell) = \mu_1$, and
$\tau\in T_{\mu}\mathcal{P}^+(M)$ fulfills $\vert \tau \vert_{\mu} = 1$ (\cite{IS6}).
Moreover, by using the normalized geometric mean ${\sigma}(\mu,\mu_1)$ we have for any $t$
\begin{equation*}
\mu(t) = a(t)\,\mu + b(t)\,\mu_1 + c(t)\,\sigma(\mu,\mu_1).
\end{equation*}
Here $a(t), b(t), c(t)$ are the functions on $[0,\ell]$, respectively, defined as follows:
\begin{equation*}
a(t) = \frac{\sin^2\frac{\ell - t}{2}}{\sin^2\frac{\ell}{2}},\quad
b(t) = \frac{\sin^2 \frac{t}{2}}{\sin^2 \frac{\ell}{2}},\quad
c(t) = 2 \cos \frac{\ell}{2} \frac{\sin\frac{\ell - t}{2}\, \sin \frac{t}{2}}{\sin^2 \frac{\ell}{2}}.
\end{equation*}
Since they satisfy $a(t), b(t), c(t) \geq 0$ and $a(t) + b(t) + c(t) = 1$,
$\mu(t)$ lies on the 2-simplex $\Delta^2(\mu,\mu_1,\sigma(\mu,\mu_1))$ and it is concluded that $\mu(t) \in {\mathcal P}^+(M)$ for any $t\in [0,\ell]$ (see Figure \ref{geodesicseg2}).
Moreover, we find easily that the function $\ell(\mu,\mu_1)$ gives arc-length of the geodesic segment(\cite{IS-7}).

\begin{remark}
The normalized geometric mean $\sigma(\mu,\mu_1)$ of two probability measures is characterized as the intersection of lines $L_1$, $L_2$ which are the tangent lines of the geodesic segment $\mu(t)$, $t\in[0,\ell]$ at the end points $\mu, \mu_1$, respectively.

Hence, 
we have
$$
\sigma(\mu,\mu_1) = \mu + \tan \frac{\ell}{2} \tau = \mu_1 + \tan \frac{\ell}{2} \tau_1,
$$where $\ell = \ell(\mu,\mu_1)$ and $\tau = {\dot\mu}(0)$, $\tau_1 = -{\dot\mu}(\ell)$
(see Figure \ref{geodesicseg2}).

\begin{figure}[htbp]
\begin{center}\begin{center}
\includegraphics[width=7cm,pagebox=cropbox,clip]{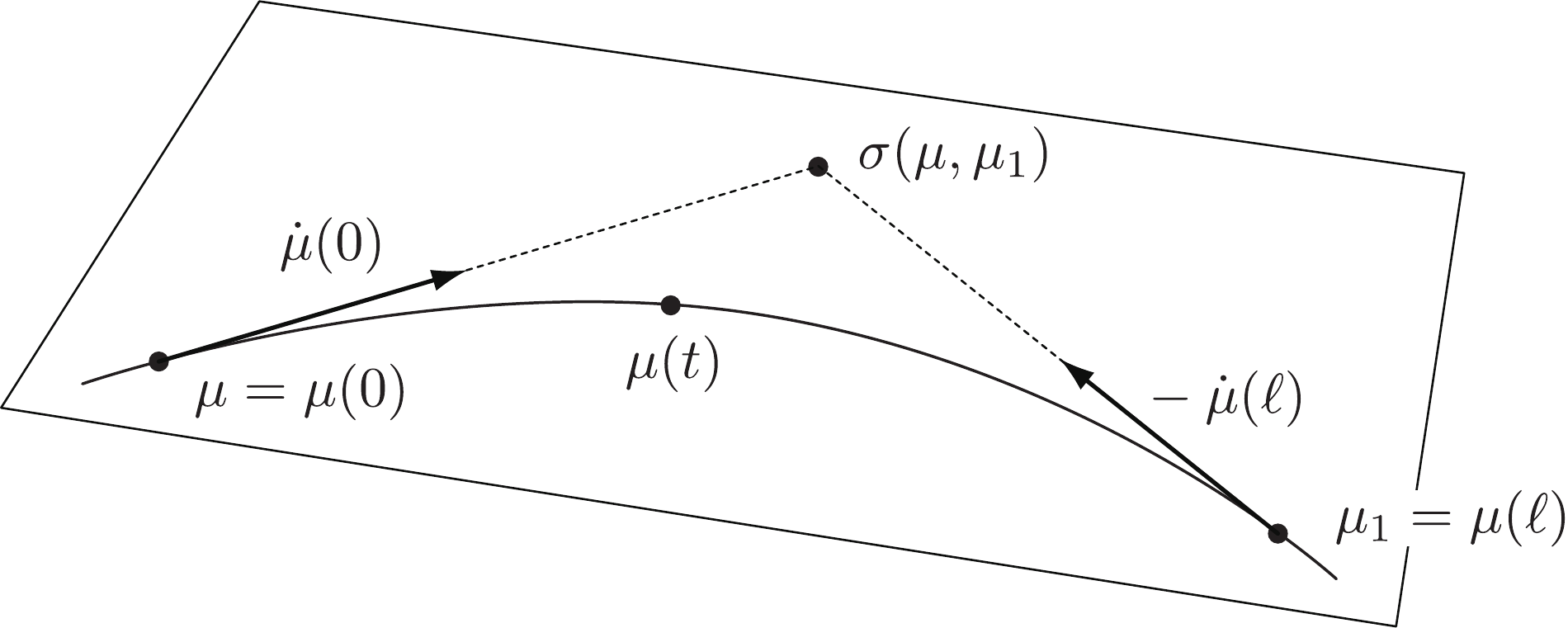}\end{center}
\caption{A geodesic segment $\mu(t)$ lies on a plane which contains the endpoints of $\mu(t)$ and their normalized geometric mean.}
\label{geodesicseg2}
\end{center}
\end{figure}
\end{remark}

\begin{remark}
Ohara \cite{O} considered a Hessian metric $g$ with respect to a certain potential on a symmetric cone $\Omega$, and defined a dualistic structure $(\nabla^{(\alpha)}, \nabla^{(-\alpha)}),\ -1\le\alpha\le 1$ on $(\Omega, g)$.
He showed that a midpoint of a geodesic segment with respect to $\nabla^{(\alpha)}$ (called an $\alpha$--geodesic segment) is an $\alpha$--power mean of endpoints.
Here the $\alpha$--power mean is an operator mean on $\Omega$ generated by a function $\displaystyle \sigma^{(\alpha)}_{1/2}(t)=\left\{(1+t^\alpha)/2\right\}^{1/\alpha}$.
In particular, the geometric mean is the $0$--power mean (see \cite{Bull}).

Incidentally, on the Riemannian manifold $(\mathcal{P}^+(M),G)$ a midpoint of a geodesic segment joining $\mu, \mu_1\in \mathcal{P}^+(M)$ is given by the normalized $1/2$\,--power mean.
 \end{remark}

\begin{theorem}\label{distance}
The Riemannian distance $d$ between probability measures $\mu$ and $\mu_1$ of $\mathcal{P}^+(M)$ is given by $d(\mu, \mu_1) = \ell(\mu,\mu_1)$.
A shortest path joining $\mu$ and $\mu_1$ is the geodesic segment $\mu(t)$, $t\in[0,\ell]$ which is given at Theorem \ref{geodesic-x}.
\end{theorem}

\begin{remark}
The arc-length function $\ell$ and the embedding $\rho$ satisfy
$$
\cos\frac{\ell(\mu,\mu_1)}{2} = 1 -\frac{1}{8} 
\vert \rho(\mu) - \rho(\mu_1)\vert_{L^2}^2.
$$
We find from this relation that for each $\mu\in \mathcal{P}^+(M)$\, the subset $W = \{ \mu_1\in \mathcal{P}^+(M)\,\vert 
\vert \rho(\mu_1) - \rho(\mu)\vert_{L^2} < \varepsilon\}$, $\varepsilon < \pi/4$ is a totally normal neighborhood of $\mu$.
Hence, we are able to define the exponential map $\exp_{\mu} : T_{\mu} {\mathcal P}^+(M) \rightarrow {\mathcal P}^+(M)$ and moreover able to discuss Gauss' lemma and certain minimizing properties of geodesics. By using these facts and properties we obtain Theorem \ref{distance}.
Moreover, we can show that the diameter of $(\mathcal{P}^+(M), G)$ is equal to $\pi$, while a proof of this fact will be elsewhere (\cite{IS-8}).
\end{remark}

\section{$\alpha$--connections and dually flat structures}\label{alpha}

In information geometry, a torsion-free affine connection $\nabla^{(\alpha)}$ parametrized by $\alpha\in\mathbb{R}$, called an $\alpha$-connection, plays an important role.
The $(+1)$-connection and the $(-1)$-connection which are flat and called the $e$-connection and the $m$-connection, respectively, are particularly important. The connections
$\nabla^{(\alpha)}$ and $\nabla^{(- \alpha)}$ are said to be dual, or adjoint to each other, with respect to the Fisher metric.
Moreover, the connection $\nabla^{(0)}$ coincides with the Levi-Civita connection of the Fisher  metric (see \cite{AN}).
As we state as follows, the notion of  $\alpha$-connection can be defined also on the space of probability measures $\mathcal{P}^+(M)$ (see \cite{I, IS-7}).

\begin{definition}Let $\alpha\in {\Bbb R}$.
Define an embedding $\rho^{(\alpha)} : \mathcal{P}^+(M) \rightarrow L^k(M, \lambda)$ by
\begin{equation*}
\mu = f(x) \lambda \mapsto\left\{
\begin{array}{ll}
\frac{2}{1 - \alpha}\,f(x)^{(1-\alpha)/2},&\alpha \not= 1,\\
\log f(x),&\alpha = 1
\end{array}
\right.
\end{equation*}
Here we set $\displaystyle{k = 2/(1-\alpha)}$ for $\alpha \not= 1$, and $k = \infty$ for $\alpha = 1$.
\end{definition}

\begin{lem}
\begin{equation}\label{alphamap}
\langle d \rho^{(\alpha)} \tau, d \rho^{(- \alpha)} \tau_1 \rangle = G_{\mu}(\tau, \tau_1),\quad
\tau, \,  \tau_1 \in T_{\mu}\mathcal{P}(M).
\end{equation}
Here $d \rho^{(\alpha)}$ is the differential map of $\rho^{(\alpha)}$ at $\mu$ and $\langle\cdot, \cdot \rangle$ is the natural pairing map $L^k(M, \lambda) \times L^{k^{\ast}}(M, \lambda) \rightarrow \, \mathbb{R}$ {\rm (} $k^{\ast} := 2/(1+\alpha)${\rm )}.
\end{lem}

\begin{definition}
Define an $\alpha$-connection $\nabla^{(\alpha)}$ for $\alpha\in{\Bbb R}$ by 
\begin{multline}\label{alphaconnection}
G_{\mu}(\nabla^{(\alpha)}_{\tau} \tau_1, \tau_2) \\
:=\left.\int_{x\in M} \, \frac{\partial^2}{\partial t \partial t_1}\, \left\{\rho^{(\alpha)}(\mu(t,t_1,t_2))\right\} \, \frac{\partial}{\partial t_2}\, \left\{\rho^{(-\alpha)}(\mu(t,t_1,t_2))\right\}\right\vert_{t=t_1=t_2=0}\, d \lambda(x).
 \end{multline}
Here $\mu(t,t_1,t_2) := \mu + t\tau + t_1\tau_1 + t_2\tau_2$.
\end{definition}

In fact, the $\alpha$-connection is presented by
$$
\displaystyle{\nabla^{(\alpha)}_{\tau} \tau_1 = - \frac{1+\alpha}{2}\,  \left(\frac{d \tau}{d \mu} \frac{d \tau_1}{d \mu}  - G_{\mu}(\tau,\tau_1) \right) \mu},
$$
from which we can assert that $\nabla^{(\alpha)}$ is torsion-free,  namely $\nabla^{(\alpha)}$ is a symmetric connection  and $\nabla^{(-1)}$ is a zero-connection and also that $\nabla^{(0)}$ is the Levi-Civita connection.
Moreover, from \eqref{alphaconnection}, we obtain the following fact:
\begin{proposition}
$\nabla^{(\alpha)}$ and $\nabla^{(- \alpha)}$ are dual each other, i.e.,
fulfill
$$
\tau\ G(\tau_1,\tau_2) = G_{\mu}(\nabla^{(\alpha)}_{\tau} \tau_1, \tau_2 ) + G_{\mu}(\tau_1, \nabla^{(-\alpha)}_{\tau} \tau_2 ).
$$
Moreover, the Riemannian curvature tensor $R^{(\alpha)}$ of $\nabla^{(\alpha)}$ is expressed by
\begin{equation}\label{riemanncurvature}
R_{\mu}^{(\alpha)}(\tau_1,\tau_2)\tau =\ \frac{1 - \alpha^2}{4}\ \left\{G_{\mu}(\tau,\tau_2)\tau_1 - G_{\mu}(\tau,\tau_1)\tau_2 \right\}.
\end{equation}
\end{proposition} 

Hence, from \eqref{riemanncurvature}, we have the following.
\begin{cor}\label{flatdual}
$\alpha$-connections on $\mathcal{P}^+(M)$ which are flat and dual each other are only the $e$-connection $\nabla^{(+1)}$ and the $m$-connection $\nabla^{(-1)}$.
\end{cor}

As shown in Theorem \ref{geodesic}, geodesics on $\mathcal{P}^+(M)$ with respect to the Levi-Civita connection admit the expression formula.
It is interesting to derive such a formula for a geodesic   $\mu(t) = f(t) \lambda,\, f(t) := f(x, t)$, $x\in M$ with respect to the $\alpha$-connection.
We follow the argument of the proof of Theorem \ref{geodesic} to obtain the following equation for a geodesic with respect to the $\alpha$-connection:
\begin{equation}\label{eq_a-geo}
\frac{\partial}{\partial t}\ \left(\frac{{\dot f}(t)}{f(t)}\right) + \frac{1-\alpha}{2} \left(\frac{{\dot f}(t)}{f(t)}\right)^2 + \frac{1+\alpha}{2} \int_{x\in M} \left(\frac{{\dot f}(t)}{f(t)}\right)^2 f(t) d\lambda = 0, 
\end{equation}
$$
\int_{x\in M} {\dot f}(t) d\lambda = 0,\qquad
{\dot f}(t) := \frac{\partial f}{\partial t}(x,t).
$$
        
\begin{proposition}     
Geodesics with respect to the $m$-connection $\nabla^{(-1)}$ coincide with affine lines defined by  $t \mapsto \mu + t\,\tau$ in $\mathcal{P}^+(M)$.
In fact, the density function given by $f(x, t) = f(x) + t\,h(x)$, $x\in M$ is a solution to \eqref{eq_a-geo}, provided $\alpha=-1$.
\end{proposition}

\begin{proposition}Let $\mu(t)$ be a geodesic with respect to the $e$-connection $\nabla^{(+1)}$ with initial conditions $\mu(0) = f(x) \lambda$, $\dot{\mu}(0) = h(x) \lambda$. Then $\mu(t)$ is represented in the form
$$
\mu(t) = \exp\left\{\int_0^t\left(\int_0^s g_0(u)du\right) ds + t\,h(x)\right\} \mu(0).
$$
Here $g_0(t)$ is a function of $t$ given by $g_0(t) = - G_{\mu(t)}(\dot{\mu}(t), \dot{\mu}(t))$.
In particular, if $\mu(0)=\mu, \mu(\ell)=\mu_1$, $\ell > 0$, then $\mu(t)$ admits the following form
$$
\mu(t)=\left\{\int_M \left(\frac{d\mu_1}{d\mu}\right)^{t/\ell}d\mu\right\}^{-1}\left(\frac{d\mu_1}{d\mu}\right)^{t/\ell}\mu
$$
(see \cite{IS-8}).
\end{proposition}

\section{Riemannian manifolds of non-positive curvature}\label{4}

In this section, we will discuss information geometry of barycenters together with the barycenter map which is another main subject of this article.
Let $X$ be a Hadamard manifold and $\partial X$ its ideal boundary.
The barycenter map $\mathrm{bar}$ is a map from $\mathcal{P}^+(\partial X)$ to $X$,
defined for $\mu\in\mathcal{P}^+(\partial X)$ by assigning to $\mu$ a critical point of $\mu$-average of the normalized Busemann function $B_{\theta} : X \rightarrow \mathbb{R}$. Thus, one has ${\mathrm{bar}}: \mathcal{P}^+(\partial X) \rightarrow X$.
Before giving the notion of barycenter and the barycenter map,
we briefly explain Hadamard manifolds, their ideal boundary and the Busemann function.

Let $(X,g)$ be a Hadamard manifold, i.e., a complete, simply connected Riemannian manifold of non-positive sectional curvature.
In what follows, we assume that $(X,g)$, simply $X$, is an $n$-dimensional Hadamard manifold.

\begin{remark}
Euclidean spaces, real hyperbolic spaces, rank one symmetric spaces of non-compact type and Damek-Ricci spaces are examples of Hadamard manifolds.
\end{remark}

Now we summarize geometric properties of Hadamard manifolds (\cite{Sa}):
\begin{enumerate}
\item there exists a unique shortest geodesic joining any two given points.
\item the distance function $x \mapsto d(x,x_0)$, $x_0\in X$ is a convex function on $X$.
Here a function $f:X\to\mathbb{R}$ is said to be convex on $X$ when for any geodesic $\gamma$ the function $t\mapsto f(\gamma(t))$ is convex.
\end{enumerate}

Now we will give definition of the ideal boundary of a Hadamard manifold.
Any geodesic on $X$ is assumed to be parametrized by arc-length.

\begin{definition}
Let $\gamma, \gamma_1: \mathbb{R} \to X$ be two geodesics on $X$.
When there exists a constant $C>0$ such that
\begin{equation}\label{96}
d(\gamma(t), \gamma_1(t))<C\quad (\forall\,t\geq 0),
\end{equation}
we say that $\gamma$ and $\gamma_1$ are asymptotically equivalent and write as $\gamma\sim\gamma_1$.
\end{definition}

The relation ``$\sim$'' gives rise to an equivalence relation on the space $\mathrm{Geo}(X)$ of all geodesic rays $\gamma : [0,\infty)\rightarrow X$.
We call the quotient space $\mathrm{Geo}(X)/\sim$ the ideal boundary of $X$, denoted by  $\partial X$.
An equivalence class represented by $\gamma\in\mathrm{Geo}(X)$ is called an asymptotic class and denoted by
$[\gamma]$ or $\gamma(\infty)$.

We denote by $S_xX$ the set of all unit tangent vectors of $X$ at $x$ and define a map $\beta_x : S_xX \rightarrow \partial X$ by $\beta_x(v):=[\gamma]$ ($\gamma(t) := \exp_x t v$).
Here $\exp_x$ is the exponential map of $X$ at $x$.
Since $X$ is of non-positive sectional curvature, $\beta_x$ is bijective.

We can define a topology on $X\cup\partial X$, called the cone topology by setting a fundamental system of neighborhoods (\cite{BGS,  EO}).
The map $\beta_x : S_xX \rightarrow \partial X$ is a homeomorphism with respect to the restriction of the cone topology.

Let $d\theta$ be a probability measure given by the normalized standard volume element of $S_xX\cong S^{n-1}$.
By the aid of the push-forward by the homeomorphism $\beta_x :S_xX\to\partial X$,
$d\theta$ induces a probability measure on $\partial X$.
We denote this probability measure by the same symbol $d\theta$.

\begin{definition}
A Hadamard manifold $X$ is said to satisfy the visibility axiom,
if the following holds(\cite{EO}):
for any $\theta, \theta_1\in\partial X$, $\theta \not= \theta_1$, there exists a geodesic $\gamma : \mathbb{R} \rightarrow X$ satisfying $[\gamma]=\theta$ and $[\gamma^{-1}]=\theta_1$,
where $\gamma^{-1}$ is the geodesic with inverse direction $\gamma^{-1}(t):=\gamma(-t)$.
\end{definition}

\begin{remark}
Rank one symmetric spaces of non-compact type, including real hyperbolic spaces, and Damek-Ricci spaces satisfy the visibility axiom.
\end{remark}

We define  the Busemann function that appears in  defining barycenters.
Let $\gamma: \mathbb{R}\to X$ be a geodesic.
Then, we define a function $b_t:X\to\mathbb{R}$,\ $t \geq 0$ by
$$
b_t(x) := d(x,\gamma(t))-t = d(x,\gamma(t)) - d(\gamma(0),\gamma(t)),
$$
where $d$ is the distance function on $X$.
Since $X$ is a Hadamard manifold,
for any $x\in X$ there exists a limit $\lim\limits_{t\to\infty}b_t(x)$, denoted by $b_{\infty}(x)$.
We call this correspondence $x\mapsto b_{\infty}(x)$ the Busemann function and denote it by $B_{\gamma}$.

A level set of the Busemann function $B_{\gamma}$ passing through $x_0\in X$,
$\{ x\in X\,|\, B_{\gamma}(x) \equiv B_{\gamma}(x_0) \}$ is called a horosphere centered at $\gamma(\infty)$, which is considered as a limit surface of geodesic spheres centered at $\gamma(t)$ with radius $d(\gamma(t),x_0)$, $t\to\infty$ (\cite{ISS}).

\begin{example}\label{busehyperplane}
The real hyperbolic plane $\mathbb{R}{\bf H}^2$ is represented by a unit disk model in the complex plane $X=\mathbb{R}{\bf H}^2=\{z\in\mathbb{C}\,\vert\,\vert z\vert<1\}$ and its ideal boundary by $ \partial X=\{z\in\mathbb{C}\,\vert\,\vert z\vert=1\}$.
Let $\gamma=\gamma(t)$ be a geodesic satisfying $\gamma(0)=0$, $[\gamma]=e^{i\,\varphi}$.
Then, the Busemann function is given by the form $B_{\gamma}(z) = \log \left(\vert z-e^{i{ \varphi}}\vert^2/(1-\vert z\vert^2)\right)$ (see \cite{Hel}).
\end{example}

In what follows, we choose a reference point $x_0\in X$ and fix it.

\begin{definition}For any $\theta\in\partial X$ let $\gamma$ be a geodesic satisfying $\gamma(0)=x_0$ and $[\gamma]=\theta$.
We denote by $B_{\theta}$ the Busemann function $B_{\gamma}$ associated with $\gamma$  and call it the normalized Busemann function with base point $x_0$.
 \end{definition}
 
Basic properties of normalized Busemann functions are summarized as follows:
\begin{enumerate}
\item $B_{\theta}(x_0)=0$ for any $\theta\in\partial X$.
\item $B_{\theta}(\gamma(t))=-t$, where $\gamma$ is a geodesic satisfying $\gamma(0)=x_0$ and $[\gamma]=\theta$.
\item $B_{\theta}$ is Lipschitz continuous.
In fact, $\vert B_{\theta}(x) - B_{\theta}(y)\vert \leq d(x,y)$ for any $x,y\in X$.
\item $B_{\theta}$ is of $C^2$-class (\cite{HI}).
\item $B_{\theta}$ admits the gradient vector field $\nabla B_{\theta}$ of $C^1$-class with unit norm $\displaystyle{\vert \nabla B_{\theta}\vert \equiv 1}$.
\item For any unit tangent vector $u\in S_xX$, there exists a unique $\theta\in\partial X$ such that $u = - (\nabla B_{\theta})_x$ .
\item $B_{\theta}$ is convex.
\item the Hessian $\nabla dB_{\theta}$ is positive semi-definite
\begin{equation}\label{107}
 (\nabla d B_{\theta})_x(u,u) \geq 0\quad(\forall u\in T_xX,\, \forall x\in X).
\end{equation}
Moreover, the Hessian satisfies $\nabla d B_{\theta}(\nabla B_{\theta}, \cdot) = 0$.
\item $\gamma \sim\gamma_1$\, $\Longleftrightarrow$\, $B_{\gamma}(\cdot) - B_{\gamma_1}(\cdot) \equiv {\rm const}$ (\cite{Sa}).
\end{enumerate}

\begin{remark}
From (iv) and (viii) $\Delta B_{\theta}(:=-{\rm trace}\nabla dB_{\theta})\leq 0$ holds for any $\theta\in\partial X$.
\end{remark}

\begin{example}
The Hessian $\nabla dB_{\theta}$ of the real hyperbolic space $\mathbb{R}{\bf H}^n$ is given by the form
\begin{equation}\label{hessianhyperbolicspace}
 (\nabla d B_{\theta})_x(u,v) = \langle u,v\rangle - \langle u, (\nabla B_{\theta})_x\rangle \langle v, (\nabla B_{\theta})_x\rangle,\quad
 u,v\in T_xX,\, x\in X
\end{equation}
(see \cite{BCG}).
 \end{example}

Let $S = S_{\theta,x}$ be the shape operator of a horosphere centered at $\theta\in\partial X$ which passes through $x\in X$.
Since the gradient vector field $\nabla B_{\theta}$ is a unit normal vector field of the horosphere, we have $(\nabla dB_{\theta})_x(u,v)= - \langle S_{\theta,x}(u),v\rangle$, where $u, v$ are vectors tangent to the horosphere (see \cite{I}).

Let $\gamma$ be a geodesic with $\theta=[\gamma]$ and let $H_t= H_{\gamma(t),\theta}$ be the horosphere centered at $\theta$ passing through $\gamma(t)$.
We find that the shape operator $S_t : T_{\gamma(t)} H_t \rightarrow T_{\gamma(t)} H_t$ of $H_t$ at $\gamma(t)$ satisfies the Riccati equation $S'_t + S_t^2 + R_t = O$ along $\gamma(t), - \infty < t <\infty$ (\cite{IS0}), where $R_t$ is the Jacobi operator defined by the Riemannian curvature tensor.

Let $\phi : X\to X$ be an isometry of a Hadamard manifold $X$.
Then, $\phi$ induces a transformation $\widehat{\phi}$ of the ideal boundary $\partial X$ as follows:
\begin{equation}\label{homeo}
\widehat{\phi}(\theta):=[\phi\circ\gamma]\quad(\forall\theta\in\partial X),
\end{equation}
where $\gamma$ is a geodesic satisfying $\gamma(0)=x_0$, and $[\gamma]=\theta$.
Then, we obtain the following.
\begin{theorem}[Busemann cocycle formula \cite{GJT}]
For any $x\in X$
\begin{equation}\label{cocycle}
B_{\theta}(\phi(x))=B_{\widehat{\phi}^{-1}(\theta)}(x)+B_{\theta}(\phi(x_0)),
\end{equation}
where $\phi^{-1}:X\to X$ is the inverse map of $\phi$.
\end{theorem}

\section{Probability measures on the ideal boundary and their barycenter}\label{barycenter}

In this section and the next section, we assume the following.
\begin{hypo}\label{hypo}
$X$ satisfies the visibility axiom and  the normalized Busemann function $\theta\mapsto B_{\theta}(x)$ is continuous as a function on $\partial X$ for any  fixed $x\in X$. 
\end{hypo}

The real hyperbolic plane satisfies Hypothesis \ref{hypo} (see Example \ref{busehyperplane}).
Damek-Ricci spaces also satisfy this hypothesis(see \cite{IS1}).

With respect to this hypothesis we have the following.
\begin{theorem}[\cite{BGS}]\label{visibility}
A Hadamard manifold $X$ satisfies the visibility axiom if and only if for any $\theta\in\partial X$ and any geodesic $\gamma$ satisfying $[\gamma]\not=\theta$ it holds $\lim_{t\to\infty}B_{\theta}(\gamma(t)) =\infty$.
\end{theorem}

\begin{definition}Let $\mu \in {\mathcal P}^+(\partial X)$ be a probability measure on $\partial X$. We define the $\mu$-averaged  Busemann function $\mathbb{B}_{\mu} : X \rightarrow \mathbb{R}$ by
\begin{equation*}
 \mathbb{B}_{\mu}(x) := \int_{\theta\in\partial X}\ B_{\theta}(x)\,d \mu(\theta),\quad
 x\in X.
\end{equation*}
\end{definition}

Under Hypothesis \ref{hypo}, 
it is shown that the averaged Busemann function $\mathbb{B}_{\mu} : X\to\mathbb{R}$ admits a minimum (see Theorem \ref{exist}).
We call this minimal point, i.e., a critical point of $\mathbb{B}_{\mu}$, a barycenter of $\mu$.
Before discussing the existence and uniqueness of the minimum of $\mathbb{B}_{\mu}$,
we state some of the main properties of averaged Busemann function $\mathbb{B}_{\mu}$.

\begin{proposition}\label{averaged}
The averaged Busemann function $\mathbb{B}_{\mu}$ has the following properties:
\begin{enumerate}
\item $\mathbb{B}_{\mu}$ is convex and $\mathbb{B}_{\mu}(x_0)=0$.
\item $\mathbb{B}_{\mu}(\gamma(t))\to\infty$ ($t\to\infty$), where $\gamma: {\Bbb R} \to X$ is an arbitrary geodesic.
\item $\mathbb{B}_{\mu}$ is Lipschitz continuous, i.e., $\vert \mathbb{B}_{\mu}(x) -  \mathbb{B}_{\mu}(y) \vert \leq d(x,y)$ for any $x, y\in X$.
\item For any $\mu\in \mathcal{P}^+(\partial X)$ and any $x\in X$, $\vert(\nabla \mathbb{B}_{\mu})_x\vert \leq 1$ holds.
\item 
The Hessian $\nabla d\mathbb{B}_{\mu}$ is represented by
\begin{eqnarray}\label{hessianaverageBusemann}
(\nabla d\mathbb{B}_{\mu})_x(u,v) = \int_{\theta\in\partial X} (\nabla d B_{\theta})_x(u,v) d \mu(\theta)\hspace{2mm}u,v\in T_xX,\, x\in X
\end{eqnarray}and hence is positive semi-definite, provided that $X$ has bounded Ricci curvature and $d \Delta B_{\theta}$ is uniformly bounded with respect to $\theta$.
\end{enumerate}
\end{proposition}

\begin{remark}\label{remarkaverageBusemann}
If $X$ is a Hadamard manifold of volume entropy $Q\geq 0$  which is asymptotically harmonic, then one has $d\Delta B_{\theta} \equiv 0$ and by using the Riccati equation $- (n-1) Q^2 \leq {\rm Ric}_x \leq 0$, while the detailed argument is omitted.  For the volume entropy and the  asymptotical harmonicity refer to Definition \ref{BusemannPoisson} and  
 Remark \ref{asymptoticallyharmonic}, \S 6, respectively.
\end{remark}

\begin{proof}\ 

(i)\,  Because $\mathbb{B}_{\mu}$ is an average of a convex function, $\mathbb{B}_{\mu}$ is also convex.
It is obvious that $\mathbb{B}_{\mu}(x_0)= 0$.

(ii)\, This will be shown later (see the proof of Theorem \ref{exist}).

(iii)\, This is obvious, since $B_{\theta}$ is Lipschitz continuous and satisfies 
$$\vert B_{\theta}(x) -  B_{\theta}(y) \vert \leq d(x,y)\ \mbox{for any $x, y\in X$}.$$

(iv)\, First, we show that $\mathbb{B}_{\mu}$ admits a gradient vector field. Let $v$ be an arbitrary tangent vector at $x\in X$ and take a geodesic $\sigma=\sigma(t)$ satisfying $\sigma(0)=x$ and $\dot{\sigma}(0)=v\in T_xX$.
We remark that $v\in T_xX$ is not necessarily a unit vector.
Since $\nabla B_{\theta}$ is uniformly bounded ($\vert\nabla B_{\theta}\vert \equiv 1$),
we may interchange the order of integration of $B_{\theta}(\sigma(t))$ with respect to $\mu\in \mathcal{P}^+(\partial X)$ and differentiation with respect to $t$. Then, the directional derivative to $v$ of  $\mathbb{B}_{\mu}$ is given by
\begin{equation}\label{119}
\begin{split}
v \mathbb{B}_{\mu}
=&\left.\frac{d}{dt}\right\vert_{t=0}\int_{\theta\in\partial X}B_{\theta}(\sigma(t))d\mu(\theta)\\
=&\int_{\theta\in\partial X}\left.\frac{\partial}{\partial t}\right\vert_{t=0}
B_{\theta}(\sigma(t))d\mu(\theta)
=\int_{\theta\in\partial X}\langle(\nabla B_{\theta})_x,v\rangle d\mu(\theta),
\end{split}
\end{equation}
from which we find that $\mathbb{B}_{\mu}$ is of $C^1$-class. This implies that the gradient vector field $\nabla\mathbb{B}_{\mu}$ is well-defined on $X$ and satisfies
\begin{equation}\label{120}
\langle(\nabla\mathbb{B}_{\mu})_x,v\rangle
=\int_{\partial X}\langle(\nabla B_{\theta})_x,v\rangle d\mu(\theta).
\end{equation}
 Letting $v=(\nabla\mathbb{B}_{\mu})_x$ in \eqref{120}, we have
\begin{equation}\label{121}
\vert(\nabla\mathbb{B}_{\mu})_x\vert^2
\leq\vert(\nabla\mathbb{B}_{\mu})_x\vert\int_{\partial X}\vert(\nabla B_{\theta})_x\vert 
d\mu(\theta){ =}\vert(\nabla\mathbb{B}_{\mu})_x\vert,
\end{equation}
from which  (iv) is shown.

(v)\, In general, the Hessian of a function $f : X\to\mathbb{R}$ is given by 
$$\nabla df(v,v)=\left.\frac{d^2}{dt^2}\right\vert_{t=0}f(\sigma(t)),$$
where $\sigma$ is a geodesic satisfying $\sigma(0) =x\in X$ and ${\dot \sigma}(0) = v$.
Therefore, to define the Hessian of the $\mu$-averaged Busemann function,
it suffices that the order of integration with respect to $\mu\in \mathcal{P}^+(\partial X)$ and differentiation with respect to $t$ in (\ref{124}), the equation below  which is obtained from \eqref{119}, is interchangeable:
\begin{equation}\label{124}
\left.\frac{d^2}{dt^2}\right\vert_{t=0}\mathbb{B}_{\mu}(\sigma(t))
=\left.\frac{d}{dt}\right\vert_{t=0}\int_{\theta\in\partial X}\langle(\nabla B_{\theta})_{\sigma(t)},
\dot{\sigma}(t)\rangle d\mu(\theta),
\end{equation}
so that it  suffices to show that 
$$
(\nabla dB_{\theta})_x(v,v) =\left.\displaystyle{\frac{d^2}{dt^2}}\right|_{t=0}B_{\theta}(\sigma(t))
$$
is uniformly bounded as a function of $\theta\in\partial X$.
From Bochner's formula (\cite[Proposition 4.15]{GHL}), uniform boundedness of $\vert \nabla d B_{\theta}\vert$ is asserted under the assumption of the boundedness of the Ricci curvature of $X$ and the uniform boundedness of $d \Delta B_{\theta}$. Therefore, one obtains (\ref{hessianaverageBusemann}).
\end{proof}

\begin{theorem}\label{exist}
Let $X$ be a Hadamard manifold satisfying Hypothesis \ref{hypo}.
Then,  an arbitrary probability measure $\mu\in\mathcal{P}^+(\partial X)$ admits a barycenter.
\end{theorem}

\begin{remark}
In \cite{BCG}, Besson et al. show the existence of barycenter. They assume that the Hadamard manifold $X$ is a rank one symmetric space of non-compact type, for which Hypothesis \ref{hypo} is satisfied.
They obtain Theorem \ref{exist} for general probability measures with no atom.
\end{remark}

We will outline a proof of Theorem \ref{exist}.
For a given constant $C>0$, we set $A_C:=\{y\in X\,\vert\,\mathbb{B}_{\mu}(y)\leq C\}$.
It is  seen that  $x_0\in A_C$, because $\mathbb{B}_{\mu}(x_0)=0$.
Hence, we find that $A_C$ is a non-empty closed subset of $X$.

Now we show that $A_C$ is bounded.
Since $\mathbb{B}_{\mu}$ is convex, $A_C$ is a convex set.

Choose $\theta\in\partial X$ arbitrarily and fix it.
Let $\gamma$ be a geodesic on $X$ satisfying $\gamma(0)=x_0$ and $[\gamma]=\theta$.
Then we will show that the convex function $\mathbb{B}_{\mu}$ satisfies $\displaystyle \lim\limits_{t\to\infty}\mathbb{B}_{\mu}(\gamma(t))= +\infty$, as follows.

Since $B_{\theta}$ is convex and $B_{\theta}(x_0)=0$, it holds that for any geodesic $\sigma$ through $x_0$ at $t=0$ ($\sigma(0) = x_0$), 
\begin{equation}\label{131}
t_1\,B_{\theta}(\sigma(t))\geq t\,B_{\theta}(\sigma(t_1))\quad(0\leq t_1\leq t)\quad\forall \theta\in\partial X.
\end{equation}
Moreover, since the $\mu$-averaged Busemann function $\mathbb{B}_{\mu}$ is also convex,  from Proposition \ref{averaged}\,(i), it holds similarly as (\ref{131})
\begin{equation}\label{convex-2}
t_1\ \mathbb{B}_{\mu}(\sigma(t))\geq t\ \mathbb{B}_{\mu}(\sigma(t_1))\quad(0\leq t_1\leq t)\quad\forall \mu\in {\mathcal P}^+(\partial X).
\end{equation}

Next, we choose an arbitrary ideal point $\theta_0\in\partial X$ and fix it.
Let $\gamma_0$ be a geodesic satisfying $\gamma_0(0)=x_0$ and $[\gamma_0]=\theta_0$.
For a positive number $t$, we set a subset $J_{\theta_0}(t)$ of $\partial X$ by $J_{\theta_0}(t):=\{\theta\in\partial X\,\vert\,B_{\theta}(\gamma_0(t))\leq 0\}$.
From the assumption of the theorem, that is, from  Hypothesis  \ref{hypo}, $B_{\theta}(x)$ is continuous as a function of $\theta$ so that $J_{\theta_0}(t)$ is a compact subset of $\partial X$.
Obviously $\theta_0\in J_{\theta_0}(t)$, since $B_{\theta_0}(\gamma_0(t)) = -t$.

\begin{lem}\label{lem_t1}
There exists $t_1 \in (0,\infty)$ such that $\mu(J_{\theta_0}(t_1)) < 1$.
\end{lem}

In fact, from \eqref{131}, we have
$J_{\theta_0}(t)\subset J_{\theta_0}(t_1)$, if  $0\leq t_1\leq t$. We have from Theorem \ref{visibility}, i.e., from the condition equivalent to the visibility axiom $\mathop{\bigcap}_{t\in[0,\infty)}J_{\theta_0}(t)=\{\theta_0\}$.
Moreover, from \eqref{21} we have for a probability measure $\mu$
\begin{equation}\label{136}
\lim_{t\to\infty}\mu(J_{\theta_0}(t))
=\mu\left(\mathop{\bigcap}_{t\in[0,\infty)}J_{\theta_0}(t)\right)=\mu(\{\theta_0\})=0,
\end{equation}
from which we obtain Lemma \ref{lem_t1}.
We notice that if $t\geq t_1$, then $\mu(J_{\theta_0}(t))\leq\mu(J_{\theta_0}(t_1))<1$ holds.

We take a compact subset $K\subset \partial X\setminus J_{\theta_0}(t_1)$ which satisfies $\mu(K)>0$.
In fact, we can choose such a subset $K$ because $\mu(\partial X \setminus J_{\theta_0}(t_1)) > 0$.
Since $B_{\theta}(\gamma(t))>0$ holds for any $\theta\in\partial X\setminus J_{\theta_0}(t)$,
we have
\begin{equation}\label{137}
\begin{split}
\int_{\theta\in\partial X}B_{\theta}(\gamma(t))d\mu(\theta)
=&\Big(\int_{\theta\in J_{\theta_0}(t)}+\int_{\theta\in\partial X\setminus J_{\theta_0}(t)}\Big)B_{\theta}(\gamma(t))d\mu(\theta)\\
\geq& \int_{\theta\in J_{\theta_0}(t)}B_{\theta}(\gamma(t))d\mu(\theta)
+\int_{\theta\in K} B_{\theta}(\gamma(t))d\mu(\theta).
\end{split}
\end{equation}
Now we fix $t_1 > 0$ in our lemma.
Since the subset $K$ is compact, we can retake the constant $C>0$ such that
\begin{equation}\label{138}
B_{\theta}(\gamma(t_1))\geq C>0\quad
(\forall\,\theta\in K).
\end{equation}
From \eqref{131}, \eqref{137} and \eqref{138}, we have
\begin{equation}\label{139}
\int_{\partial X} B_{\theta}(\gamma(t)) d\mu(\theta) \geq \frac{t}{t_1} \int_{J_{\theta_0}(t)} B_{\theta}(\gamma(t_1)) d\mu(\theta) + C \frac{t}{t_1} \mu(K).
\end{equation}
To estimate the first term of the right hand side of \eqref{139},
we set $D:=\sup\{\vert B_{\theta}(\gamma(t_1))\vert$\, $\,\,\vert\,\theta\in\partial X\}$.
Since $\partial X$ is compact, the continuous function $\theta\mapsto B_{\theta}(\gamma(t_1))$ is bounded as a function of $\theta$.
Hence, $D<\infty$ is asserted.
From \eqref{139}, we get the following:
\begin{equation}\label{141}
\mathbb{B}_{\mu}(\gamma(t)) =
\int_{\partial X}B_{\theta}(\gamma(t))d\mu(\theta)\geq\frac{t}{t_1}\left(C\mu(K)-D\mu(J_{\theta_0}(t))\right).
\end{equation}
Therefore, we obtain from \eqref{136} $\displaystyle \lim_{t\to\infty}\mathbb{B}_{\mu}(\gamma(t)) =+\infty$.

Now, we suppose that $A_C$ be not bounded.
Then, we can take a sequence $\{ y_i\}$ of points in $A_C$ such that $d_i = d(y_i,x_0) \to +\infty$, $i \to +\infty$ ($d_i < d_j$, $i < j$).
Therefore, there exists a  sequence $\{v_i\}$ of unit tangent vectors such that $y_i = \exp_{x_0} d_i v_i$.
Take an appropriate subsequence of $\{v_i\}$, denoted by the same letter $\{v_i\}$ for brevity and set $\lim_i v_i = v_{\infty}$.
A geodesic $\gamma_{\infty}(t)$ given by $\gamma_{\infty}(t)=\exp_{x_0} t v_{\infty}$ satisfies $\displaystyle \mathbb{B}_{\mu}(\gamma_{\infty}(t))  = \lim_{i\to\infty} \mathbb{B}_{\mu}(\exp_{x_0} t v_i)$, $t > 0$.
Since for any $t > 0$ there exists $i_0 > 0$ such that $d_i \geq t$ for any $i \geq i_0$,
we have from \eqref{convex-2}
\begin{equation}
\mathbb{B}_{\mu}(\exp_{x_0} t v_i) \leq \frac{t}{d_i}\, \mathbb{B}_{\mu}(\exp_{x_0} d_i v_i) = \frac{t}{d_i} \mathbb{B}_{\mu}(y_i) \leq \frac{t}{d_i} C \leq C,
\end{equation}
a contradiction, since $\mathbb{B}_{\mu}(\gamma_{\infty}(t)) \to \infty$ ($t\to \infty$).
Hence, $A_C$ is bounded.
Since $A_C$ is bounded and closed, $\mathbb{B}_{\mu}$ has a minimum value on $A_C$, i.e., $\mu$ has a barycenter.

With respect to the uniqueness of barycenter, one has the following.

\begin{theorem}
Let  $X$ be a Hadamard manifold satisfying Hypothesis \ref{hypo}.
If for any $\mu\in\mathcal{P}^+(\partial X)$ the Hessian of $\mu$-averaged Busemann function is positive definite everywhere on $X$, then there exists uniquely a barycenter of $\mu$.
\end{theorem}

For the positive definiteness of the Hessian of $\mu$-averaged Busemann function, we have the following.
\begin{theorem}\label{hessian}Assume that every averaged Busemann function admits its Hessian in the form (\ref{hessianaverageBusemann}) in (v), Proposition \ref{averaged}.
\begin{enumerate}
\item If the Ricci curvature of $X$ is negative everywhere, i.e., $\mathrm{Ric}_x < 0$, $\forall x\in X$, then for any $\mu\in\mathcal{P}^+(\partial X)$ the Hessian $\nabla d\mathbb{B}_{\mu}$ of the $\mu$-averaged Busemann function is positive definite everywhere on $X$.
\item If there exists a certain probability measure  $\mu_0\in\mathcal{P}^+(\partial X)$ such that the Hessian $\nabla d\mathbb{B}_{\mu_0}$ of $\mathbb{B}_{\mu_0}$ is positive definite on $X$, 
the Hessian $\nabla d\mathbb{B}_{\mu}$ of the $\mu$-averaged Busemann function $\mathbb{B}_{\mu}$ is also positive definite for any $\mu\in\mathcal{P}^+(\partial X)$. 
\end{enumerate}
\end{theorem}

\begin{proof}
We will give an outline of a proof of (i).
Assume that there exists a $\mu \in {\mathcal P}^+(\partial X)$ such that for $x\in X$ and $u(\not=0)\in T_xX$,  $(\nabla d \mathbb{B}_{\mu})_x(u,u) = 0$ holds.
Since  the Hessian $\nabla d B_{\theta}$ in (\ref{hessianaverageBusemann}) is positive semi-definite for any $\theta\in\partial X$, we have then $(\nabla d B_{\theta})_x(u,u) = 0$.
Hence we see ${ \langle S_{{\theta,x}} u,u\rangle} = 0$.
Here $S_{{ \theta,x}}$ is the shape operator of the horosphere $H_{x,\theta}$ centered at $\theta$ passing through $x\in X$.
Since $S_{\theta,x}$ is negative semi-definite, we have $\langle S_{{\theta,x}} u,v\rangle = 0$ for any $v\in T_xX$.
In fact, if we assume that there exists a tangent vector $v$ such that $\langle S_{{\theta,x}} u,v\rangle \not= 0$, $\langle S_{{\theta,x}} (u+tv),(u+tv)\rangle $ takes a positive value for some $t\not=0$, which is a  contradiction, since $S_{\theta,x}$ is negative semi-definite.  Hence, we have $S_{\theta,x}u = 0$.
From the Riccati equation for $S_t$, we can show that the sectional curvature of any 2-plane which contains $u$ vanishes so that we obtain $\mathrm{Ric}_x(u) = 0$.
This is a contradiction, since the Ricci curvature is assumed to be negative. Hence we obtain (i).

 To show (ii) we set $C = \min \left\{\left. (d\mu/d\mu_0)(\theta)\,\right|\, \theta\in \partial X\right\}$. Then $C > 0$ from the compactness of $\partial X$.
From the positive semi-definiteness of the Hessian $\nabla d B_{\theta}$ we find that $(\nabla d \mathbb{B}_{\mu})_x(u,u) \geq C (\nabla d \mathbb{B}_{\mu_0})_x(u,u)$, $\forall u\in T_xX$, $\forall x\in X$, from which (ii) is obtained.
\end{proof}

\begin{cor}
Under Hypothesis \ref{hypo}, namely, the assumption of Theorem \ref{exist} and  either (i) or (ii) of Theorem \ref{hessian}, an arbitrary probability measure
$\mu\in\mathcal{P}^+(\partial X)$ has a unique barycenter.
\end{cor}

\begin{remark}\label{basepoint}
The barycenter of the standard probability measure $d \theta\in \mathcal{P}^+(\partial X)$ is the base point $x_0\in X$ (see \cite{IS6}).
\end{remark}

\section{Barycenter maps and barycenterically associated maps}\label{barycentermap}

When there exists a unique barycenter $x=\mathrm {bar}(\mu)\in X$ for any $\mu\in\mathcal{P}^+(\partial X)$,
we can define a map ${\rm bar}:\mathcal{P}^+(\partial X)\to X$; $\mu \mapsto x$ and call it barycenter map.

We notice that the differential map $(d {\rm bar})_{\mu} : T_{\mu} \mathcal{P}(\partial X) \rightarrow T_xX$,\, $x ={\rm bar}(\mu)$ of the barycenter map $\mathrm{bar}$ is well defined and is surjective (\cite{I}).

\begin{definition}\label{linear}
Let ${\rm bar}^{-1}(x)$ be the inverse image of the map $\mathrm{bar}$ at a point $x$.
For any $\mu\in\mathrm{bar}^{-1}(x)$ we define a map $\nu_x^{\mu}:T_xX\to T_{\mu}\mathcal{P}^+(\partial X)$ by
\begin{equation}\label{linearmap}
\nu_x^{\mu}(u)(\theta):=(dB_{\theta})_x(u)\,d\mu(\theta)\qquad
(u\in T_xX,\hspace{2mm}\theta\in\partial X).
\end{equation}
(see \cite{IS5}).
\end{definition}
Notice that the image of the map $\nu_x^{\mu}$ is included in the tangent space  $T_{\mu}\mathcal{P}^+(\partial X)$ at $\mu$, since  $x$ is the barycenter of $\mu$ if and only if 
$$
(d \mathbb{B}_{\mu})_x(u) = \int_{\theta\in\partial X} (dB_{\theta})_x(u)\,d\mu(\theta) = 0\qquad
$$ holds for any $ u\in T_xX$. Moreover, we find that $\nu_x^{\mu}$ is injective (\cite{I}).

Let $\mu(t)$ be a curve in ${\rm bar}^{-1}(x)$, $\vert t\vert < \varepsilon$ satisfying $\mu(0) = \mu$ and ${\dot \mu}(0) = \tau \in T_{\mu}\mathcal{P}^+(\partial X)$.
Then, $\tau$ fulfills $\int_{\theta\in\partial X} (dB_{\theta})_x(u) d \tau(\theta) = 0$ for any $u\in T_xX$, i.e., $G_{\mu}(\tau, \nu_x^{\mu}(u)) = 0$ with respect to the Fisher metric $G$,
from which we find that tangent vectors of ${\rm bar}^{-1}(x)$ are orthogonal to the image of $\nu_x^{\mu}$.
The following proposition indicates that ${\rm bar} : \mathcal{P}^+(\partial X) \rightarrow X$ admits a fiber space structure in a tangent space level.

\begin{proposition}\label{fibre}
Let $x\in X$.
For each $\mu\in{\rm bar}^{-1}(x)$, $T_{\mu}\mathcal{P}^+(\partial X)$ is decomposed into the $G_{\mu}$-orthogonally direct sum:
\begin{equation}
T_{\mu}\mathcal{P}^+(\partial X)
= T_{\mu}{\rm bar}^{-1}(x) \oplus {\rm Im}\,\nu_x^{\mu}.
\end{equation}
\end{proposition}

\begin{remark}
$T_{\mu}{\rm bar}^{-1}(x)\subset T_{\mu}\mathcal{P}^+(\partial X)$ is the vertical subspace of the fiber space structure.
On the other hand, ${\rm Im}\,\nu_x^{\mu}$ is the horizontal subspace, normal to the fibers,
which means that the barycenter map $\mathrm{bar}$ satisfies an infinitesimal trivialization.
This proposition suggests that the barycenter map $\mathrm{bar}$ itself satisfies a local trivialization.
\end{remark}
 
 Now we are going to see geometric properties of a barycenter map.

\begin{proposition}\label{cocycle2}
Let $\phi$ be an isometry of $X$ and $\widehat{\phi}:\partial X\to\partial X$ be a homeomorphism of $\partial X$ induced by $\phi$\, (see \eqref{homeo}).
Let $\widehat{\phi}_{\sharp}:\mathcal{P}^+(\partial X)\to\mathcal{P}^+(\partial X)$ be the push-forward of $\widehat{\phi}:\partial X\to\partial X$. 
Then, we have for any $ \mu\in \mathcal{P}^+(\partial X)$\hspace{2mm} 
$\mathrm{bar}(\widehat{\phi}_{\sharp}\mu)=\phi({\rm bar}(\mu))$, i.e.,
$\mathrm{bar}\circ\widehat{\phi}_{\sharp}=\phi\circ{\rm bar}$ holds.
\end{proposition}

In fact, by integrating both side of the Busemann cocycle formula \eqref{cocycle} with respect to $\mu$, we obtain the averaged-Busemann cocycle formula
$$
\mathbb{B}_{\mu}(\phi^{-1}x) = \mathbb{B}_{\widehat{\phi}_{\sharp}\mu}(x) + \mathbb{B}_{\mu}(\phi^{-1}x_0),\quad\forall \mu\in {\mathcal P}^+(\partial X),
$$
from which our proposition is obtained.

Poisson kernel probability measures are measures significantly important in considering the barycenter maps.
As mentioned in \S \ref{intro},
the Poisson kernel $P(x,\theta)$ is the fundamental solution of the Dirichlet problem at infinity.
We define the Poisson kernel based on the argument of Shoen-Yau given in \cite{SY} as follows:

\begin{definition}[\cite{SY}]\label{poisson}
Let $x_0\in X$ be a fixed point and let $P(x,\theta)$ be a function on $X \times \partial X$. Then, $P(x,\theta)$ is called the Poisson kernel, normalized at a base point $x_0$, when it  satisfies  the following conditions:
\begin{enumerate}
\item $\Delta P(\cdot,\theta) = 0\hspace{3mm}\forall\,\theta\in\partial X$.
\item $P(x_0,\theta) = 1\hspace{3mm}\forall\,\theta\in\partial X$.
\item $P(x,\theta) d\theta \in \mathcal{P}^+(\partial X)$\hspace{3mm} $\forall\,x\in X$.
\item for any fixed $\theta\in\partial X$, the function on $X$, $x\mapsto P(x,\theta)$ is extended to a continuous function on $X\cup\partial X\setminus\{\theta\}$ and satisfies
$\lim_{x\to\theta'}P(x,\theta)=0$\hspace{3mm}for any $\theta'$ satisfying $\theta'\not=\theta$.
\end{enumerate}
\end{definition}

For any $f\in C^0(\partial X)$, the function $u(x)=\int_{\theta\in\partial X}P(x,\theta)\,f(\theta)\,d\theta$ is a solution to $\Delta u=0$ and satisfies the boundary condition at infinity $u\vert_{\partial X} = f$ (see \cite{SY,IS1}).

\begin{remark}Under the negative sectional curvature condition $- b^2 \leq   K \leq - a^2  < 0$, the existence and uniqueness of the Poisson kernel is guaranteed (see \cite{SY, AS}).
Refer also to \cite{An,Su} for the Dirichlet problem at infinity.
\end{remark}

\begin{definition}\label{BusemannPoisson}
If the Poisson kernel $P(x,\theta)$ is represented in the form $P(x,\theta)=\exp(-Q\,B_{\theta}(x))$,
we call it the Busemann--Poisson kernel,
where $Q= Q(X):=\lim\limits_{r\to\infty}\frac{1}{r}\log\,{\rm Vol}(B(x,r))$ is a constant, called the volume entropy of $X$  which represents exponential volume growth of geodesic spheres.
\end{definition}

 \begin{remark}
The family of  all Damek-Ricci spaces, which is a class of Hadamard manifolds, includes rank one symmetric spaces of non-compact type. Any Damek-Ricci space carries the Busemann--Poisson kernel. For this see \cite{IS1}.
\end{remark}

\begin{remark}\label{asymptoticallyharmonic}
It is necessary for a Hadamard manifold $X$ to carry the Busemann--Poisson kernel that
$X$ is asymptotically harmonic, i.e., any horosphere has a constant mean curvature $-Q$.
Moreover, from Definition \ref{poisson}, (iii) and Theorem \ref{visibility},
the Hadamard manifold $X$ also satisfies the visibility axiom.
\end{remark}
\begin{proposition}\label{Busemann-Poisson}
If the Poisson kernel $P(x,\theta)$ is the Busemann--Poisson kernel, in particular,
then the barycenter of the probability measure $\mu_x:=P(x,\theta)\,d\theta\in\mathcal{P}^+(\partial X)$ defined by $P(x,\theta)$, parametrized by $x\in X$ is the point $x$.
Hence, $\mathrm{bar} : \mathcal{P}^+(\partial X) \rightarrow X$ is surjective (see \cite{BCG, I}).
\end{proposition}
\begin{theorem}[\cite{I}]
For the Busemann--Poisson kernel probability measure $\mu_x$, $x\in X$, the Hessian of  $\mu_x$-averaged Busemann function $(\nabla d{\Bbb B}_{\mu_x})_y$, $y\in X$ is positive definite.\end{theorem}
In fact, we have
$$
(\nabla d {\Bbb B}_{\mu_x})_x(u,u) = Q\,G_{\mu_x}(\nu_x^{\mu_x}(u), \nu_x^{\mu_x}(u)),\quad
u\in T_xX,\ x\in X.
$$ From an argument similar to the proof of Theorem \ref{hessian} (ii), we have for any $y \in X$
 $$
(\nabla d {\Bbb B}_{\mu_y})_x(u,u) > 0
,\quad
u\in T_xX,\ u\not=0,\, x\in X.
$$ 

From the above theorem, if a Hadamard manifold $X$ carrying the Busemann--Poisson kernel satisfies the assumption of Theorem \ref{exist}, namely Hypothesis \ref{hypo}, then, the function $\theta\mapsto B_{\theta}(\cdot)$ is continuous on $\partial X$ so that from Theorem \ref{hessian} (ii), any $\mu\in\mathcal{P}^+(\partial X)$ has a unique barycenter.

\begin{remark}
Let $\mu,\, \mu_1\in{\rm bar}^{-1}(x)$.  Then the path $\mu(t):=(1-t)\mu+t\mu_1$ belongs to ${\rm bar}^{-1}(x)$\hspace{2mm} for any $t\in[0,1]$. The curve $t \mapsto \mu(t)$, $t\in[0,1]$ yields a geodesic with respect to the $m$-connection $\nabla^{(-1)}$.
\end{remark}

\begin{proposition}\label{belonging}[\cite{I}]
Let $t\mapsto\mu(t)\in\mathcal{P}^+(\partial X)$ be a geodesic with respect to the Levi-Civita connection.
Then, for any $t$\,  $\mu(t)\in\mathrm{bar}^{-1}(x)$ if and only if the following are fulfilled:
\begin{enumerate}
\item $\mu(0)\in{\rm bar}^{-1}(x)$.
\item $\dot{\mu}(0)\in T_{\mu(0)}{\rm bar}^{-1}(x)$.
\item $h_{\mu(0)}(\dot{\mu}(0),\dot{\mu}(0))=0$, where $h$ is the second fundamental form of the submanifold $\mathrm{bar}^{-1}(x)$,
$$
h_{\mu} : T_{\mu}{\rm bar}^{-1}(x) \times T_{\mu}{\rm bar}^{-1}(x) \rightarrow N_{\mu} = \mathrm{Im}\, \nu_x^{\mu};\ (\alpha,\beta) \mapsto (\nabla_{\alpha} \beta)^{\perp}.
$$
Here $(\nabla_{\alpha}\beta)^{\perp}$ is the normal component of a vector $\nabla_{\alpha}\beta$.
\end{enumerate}
\end{proposition}

\begin{remark}
The barycenter of the standard probability measure $\lambda = d \theta$ is the base point $x_0$ (Remark \ref{basepoint}).
By the identification $\partial X\cong S_{x_0}X$, we set $\tau = q(v) d\theta\, \in T_{d\theta}\mathcal{P}^+(\partial X)$ for an arbitrary unit tangent vector $v \in S_{x_0}X$ by $q(v)= v^i v^j$,\, $i,j = 1,\cdots,n$, $i \not= j$.
Here ${\{v^i, 1\leq i\leq n\}}$ is the components of $v$ with respect to a certain orthonormal basis ${\{\bf e}_i\}$ of $T_{x_0}X$; $v = \sum_{i=1}^n v^i {\bf e}_i$.
Since $\tau$ satisfies (ii) and (iii) of Proposition \ref{belonging}, the geodesic with initial velocity vector $\tau$ belongs to $\mathrm{bar}^{-1}(x_0)$.
\end{remark}

\begin{theorem}\label{geodesicinfiber}
Let $\mu$, $\mu_1 \in {\rm bar}^{-1}(x)$, $x\in X$.
A geodesic with respect to the Levi-Civita connection joining $\mu$ and $\mu_1$ is contained in $\mathrm{bar}^{-1}(x)$ if and only if $\sigma(\mu,\mu_1)\in \mathrm{bar}^{-1}(x)$ holds.
Here ${\sigma}(\mu,\mu_1)$ is the normalized geometric mean (see Definition \ref{geometricmean} in \S 2.2).
\end{theorem}

Theorem \ref{geodesicinfiber} is immediate from Theorem \ref{geodesic-x},  which expresses the representation formula for a geodesic in terms of the normalized geometric mean.
See \cite{IS-8} for details.

Let $\Phi : \partial X\to\partial X$ be a homeomorphism and $\Phi_{\sharp}:\mathcal{P}^+(\partial X)\to\mathcal{P}^+(\partial X)$ be the push-forward by $\Phi$.
We recall that $\Phi_{\sharp}$ is isometric with respect to the Fisher metric $G$ (see Theorem \ref{pushisometry} in \S \ref{2.2}).

\begin{definition}[\cite{I}]
Let $\phi:X\to X$ be a bijective map.
We call $\phi$ a map barycentrically associated to $\Phi$, when $\phi$ satisfies $\mathrm{bar}\circ\Phi_{\sharp}=\phi\circ\mathrm{bar}$.
\end{definition}

Let $\phi$ be an isometry of $X$ and $\widehat{\phi}:\partial X\to\partial X$ be the homeomorphism induced by $\phi$.
Then, $\phi$ is a map barycentrically associated to $\widehat{\phi}$ (see Proposition \ \ref{cocycle2}).

\begin{definition}
We define a map $\Theta:X\to\mathcal{P}^+(\partial X)$ by $\Theta(x) := \mu_x (\,= P(x,\theta)\,d\theta\,)$, $x\in X$.
We call this map the Poisson kernel map.
\end{definition}
In what follows, we assume that a Hadamard manifold $X$ satisfies Hypothesis \ref{hypo} and carries the Busemann-Poisson kernel.
Then, we recall that $\mathrm{bar}(\mu_x) = x$ (see Proposition \ref{Busemann-Poisson}).
Hence, the map $\Theta$ satisfies $\mathrm{bar}\circ\Theta={\rm id}_X$, i.e., $\Theta$ is a section of the projection $\mathrm{bar} : \mathcal{P}^+(\partial X) \rightarrow X$.
Its differential map satisfies $(d\Theta)_x = - Q\, \nu_x^{\mu_x}$ for any $x\in X$.

We can show that any isometry $\varphi$ of $X$ satisfies $\Theta\circ \varphi = {\hat\varphi}_{\sharp}\circ \Theta$.
In relation with this fact, we have the following theorem.
\begin{theorem}[\cite{IS4, IS6, IS5}]\label{Poissonmap}
Let $\Phi : \partial X\to\partial X$ be a homeomorphism of $\partial X$ and $\phi : X \rightarrow X$ be a bijective $C^1$-map which is barycentrically associated to $\Phi$.
If the maps  $\Phi$ and $\phi$ admit a relation $\Theta\circ\phi = \Phi_{\sharp}\circ\Theta$,
then $\phi$ turns out to be an isometry of $X$.
Moreover, the transformation $\hat{\phi}$ of $\partial X$ canonically induced by $\phi$ coincides with $\Phi$.
\end{theorem}
\begin{remark}
If we assume singly $\Theta\circ\phi=\Phi_{\sharp}\circ\Theta$ ,
it follows then that $\mu_{\phi(x)}=\Theta(\phi(x))=\Phi_{\sharp}\Theta(x)=\Phi_{\sharp}\mu_x$  
and hence $\phi(x)={\rm bar}(\mu_{\phi(x)})={\rm bar}(\Phi_{\sharp}\mu_x)$ from which
we find that the barycenter of $\Phi_{\sharp}\mu_x$ is $\phi(x)$.
\end{remark}

\begin{remark}
In general, we define for a homeomorphism $\Phi:\partial X\to\partial X$ a map $\phi : X\to X$ by $\phi(x):={\rm bar}(\Phi_{\sharp}\mu_x)$,\, $x\in X$. It is shown that when $X$ is a real hyperbolic space,  the map $\phi$ is of $C^1$-class and satisfies
$$
\mathrm{Jac}\,\phi_x :=\sqrt{{\rm det}((d \phi_x)^{\ast}(d \phi_x))} \leq 1.
$$
 Moreover, equality in the above holds if and only if the differential map $d \phi_x:T_xX\to T_{\phi(x)}X$ is a linear isometry (see \cite{BCG, BCG2, Sau, N}).
\end{remark}

We can  weaken the assumption of Theorem \ref{Poissonmap} in the following way.
\begin{theorem}[\cite{IS6,I}]
Let $\Phi : \partial X \rightarrow \partial X$ be a homeomorphism of $\partial X$ and $\varphi : X \rightarrow X$ be a map of $C^1$-class which is barycentrically associated to $\Phi$.
Suppose that there exists a section $\Sigma : X \rightarrow \mathcal{P}^+(\partial X)$ of the fiber space structure induced by the barycenter map $\mathrm{bar} : \mathcal{P}^+(\partial X) \rightarrow X$, i.e., a map $\Sigma$ satisfying ${\rm bar}\circ \Sigma = \mathrm{id}_{X}$, such that each of the following two diagrams commutes:\\
\begin{minipage}{0.5\hsize}
$$
\begin{CD}
\mathcal{P}^+(\partial X) @>\Phi_{\sharp} >> \mathcal{P}^+(\partial X) \\
@A \Sigma AA @AA \Sigma A \\
X @>>\varphi >X
\end{CD}
$$
\end{minipage}
\begin{minipage}{0.4\hsize}
$$
\begin{CD}
T_{\mu_x}\mathcal{P}^+(\partial X)@> { (d}\Phi_{\sharp}{ )_{\mu_x} } >> T_{\mu_{\varphi { (}x{ )} }}\mathcal{P}^+(\partial X) \\
@A \nu_x^{\mu_x} AA @AA \nu_{\varphi{ (} x{ )} }^{\mu_{\varphi { (}x{ )} }} A\\
T_xX @>> d \varphi_x> T_{\varphi { (}x{ )} }X
\end{CD}
$$
\end{minipage}\\
Here $\mu_x := \Sigma(x)\in\mathcal{P}^+(\partial X)$.
Then, $\varphi$ is an isometry of $X$ and the transformation $\widehat{\varphi}$ of $\partial X$ induced by $\varphi$ coincides with $\Phi$.
\end{theorem}
This theorem is shown by using geometric properties of the normalized Busemann function.
Refer to \cite{IS6} for details.
Theorem \ref{Poissonmap} is also obtained as a corollary of this theorem.

\section*{Acknowledgments}

The authors would like to appreciate Professor Naoyuki Koike since he offered an opportunity for writing a manuscript.
The authors also wish to thank the referees for their careful refereeing and many useful and valuable comments.


\end{document}